\documentclass{article}
\usepackage[cp1251]{inputenc}
\usepackage[english]{babel}
\usepackage{amsmath}
\usepackage{amsthm}
\usepackage{mathtext}
\usepackage{amssymb}
\usepackage{amscd}
\usepackage{graphics}
\usepackage{euscript}
\usepackage{bbm}

\theoremstyle{plain}
\newtheorem{theo}{Theorem}
\newtheorem*{theo*}{Theorem}
\newtheorem{prop}[theo]{Proposition}
\newtheorem{cor}[theo]{Corollary}
\newtheorem*{prop*}{Proposition}
\newtheorem{lem}[theo]{Lemma}
\theoremstyle{definition}
\newtheorem{defin}[theo]{Definition}

\newtheorem{problem}[theo]{Problem}

 \DeclareMathOperator{\s}{s_{\mathbb R}}
\DeclareMathOperator{\Z}{\mathcal{Z}}\DeclareMathOperator{\R}{\mathbb R\mathcal{Z}}
\DeclareMathOperator{\z}{\mathbb Z}
\renewcommand{\Im}{\rm Im\,}

\renewcommand{\S}{S^0}
\renewcommand{\phi}{\varphi}
\renewcommand{\o}{\omega}
\renewcommand{\a}{\textit{\textbf{a}}}

\renewcommand{\L}{\rm Ls}

\newcommand{\e}{\textit{\textbf{e}}}

\DeclareMathOperator{\x}{\textit{\textbf{x}}}

\newcommand {\ib}[1]{\textit{\textbf{#1}}}
\newcommand {\mat}{\rm Mat}

\begin{document}
\title{\bf\Large{Criterion for the Buchstaber invariant of simplicial complexes to be equal to two.}}
\author{Nickolai Erokhovets \thanks{The work was partially supported by Dmitry Zimin's "Dynasty"
foundation, by the Russian Government project 11.G34.31.0053, by the Russian
President grants MD-2253.2011.1 and NS-4995.2012.1, and by the RFBR grants No
11-01-00694-a and 12-01-92104-JF-a}}
\date{}
\maketitle

\begin{abstract} In this paper we study the Buchstaber
invariant of simplicial complexes, which comes from toric topology. With each
simplicial complex $K$ on $m$ vertices we can associate a moment-angle
complex $\mathcal Z_K$ with a canonical action of the compact torus $T^m$.
Then $s(K)$ is the maximal dimension of a toric subgroup that acts freely on
$\mathcal Z_K$. We develop the Buchstaber invariant theory from the viewpoint
of the set of minimal non-simplices of $K$. It is easy to show that $s(K)=1$
if and only if any two and any three minimal non-simplices intersect. For
$K=\partial P^*$, where $P$ is a simple polytope, this implies that $P$ is a
simplex. The case $s(P)=2$ is such more complicated. For example, for any
$k\geqslant 2$ there exists an $n$-polytope with $n+k$ facets such that
$s(P)=2$. Our main result is the criterion for the Buchstaber invariant of a
simplicial complex $K$ to be equal to two.
\end{abstract}

\section{Introduction.}
For the introduction to toric topology see \cite{BP02}. \emph{Moment-angle
space} is a key notion of toric topology. It was introduced by M.~Davis and
T.~Januszkiewicz in \cite{DJ91}. In our paper we use the following
construction (see \cite{BP02}).

Let $K=\{\sigma\subset [m]=\{1,2,\dots,m\}\}$ be a simplicial complex on $m$
vertices. For the pair of topological spaces $(X,A)$, $A\subseteq X$, define
the $K$-power as
$$
(X,A)^K=\{(x_1,\dots,x_m)\in X^m\colon \{i:x_i\notin A\}\in K\}.
$$
In particular cases $(D^2,S^1)$, where $D^2=\{\ib{z}\in\mathbb C\colon
|z|\leqslant 1\}$, $S^1=\{\ib{z}\in\mathbb C\colon |z|= 1\}$, and
$(D^1,S^0)$, where $D^1=\{x\in\mathbb R\colon |x|\leqslant 1\}$,
$\S=\{\pm1\}$, we obtain a \emph{moment-angle complex} $\Z_K$ and a
\emph{real moment-angle complex} $\R_K\subset\Z_K$.

There are canonical coordinate actions of $T^m=(S^1)^m$ on $\Z_K$, and
$(\S)^m$ on $\R_K$. We will use the isomorphisms
\begin{gather*}
(\mathbb R^m/\mathbb Z^m)\simeq T^m\colon (\phi_1,\dots,\phi_m)\to(e^{2\pi
i \phi_1},\dots,e^{2\pi i \phi_m}), \text{ and }\\
\z_2^m\simeq (\S)^m\colon
(\alpha_1,\dots,\alpha_m)\to((-1)^{\alpha_1},\dots,(-1)^{\alpha_m}), \text{
where }\z_2=\{0,1\}.
\end{gather*}

For the simplex $\sigma\in K$ define the coordinate subgroup
$$
T^{\sigma}=\{(t_1,\dots,t_m)\in T^m\colon\{i\colon t_i\ne 1\}\subset
\sigma\}.
$$
Let $\ib{x}=(x_1,\dots,x_m)\in\Z_K$. Define $\sigma(\ib{x})=\{i\in[m]\colon
x_i=0\}\in K$. Then the stabilizer $T^m_{\x}$ of the point $\x$ is
$T^{\sigma(\ib{x})}$, and $K=\{\sigma(\x)\colon \x\in\Z_K\}$.

\begin{defin}
A \emph{Buchstaber invariant} $s(K)$ is the maximal dimension $s$ of the
toric subgroup $H\subset T^m$, $H\simeq T^s$, that acts freely on $\Z_K$. A
\emph{real Buchstaber invariant} $\s(K)$ is the maximal dimension $s$ of the
subgroup $H_2\subset \z_2^m$ that acts freely on $\R_K$.
\end{defin}
For a simple polytope $P$ define $s(P)$ as the Buchstaber invariant $s(K)$ of
the boundary complex $K=\partial P^*$ of the polar simplicial polytope.
Similarly, $\s(P)=\s(\partial P^*)$.

If the subgroup $H\subset T^m$ acts freely on $\Z_K$, then the subgroup
$H_2=H\cap(S^0)^m$ acts freely on $\R_K$, therefore $s(K)\leqslant \s(K)$ for
all $K$. It can be shown that $\s(K)\leqslant m-\dim K-1$.

\begin{problem}[Victor\,M. Buchstaber, 2002] To find an EFFECTIVE
combinatorial description of $s(K)$.
\end{problem}

The Buchstaber invariant has been studied since 2001. The problem was
originally formulated and studied for simple polytopes. In the case of simple
$n$-polytope $P$ with $m$ facets we have $1\leqslant s(P)\leqslant m-n$.
I.~Izmestiev \cite{Iz01a,Iz01b} proved the estimate $s(P)\geqslant
m-\gamma(P)$, where $\gamma(P)$ is the chromatic number of $P$, and found the
lower bound in terms of the group of projectivities (see \cite{Jo01}) of $P$.
The case of simplicial complexes that are skeleta of a simplex was considered
by M.~Masuda and Y.~Fukukawa \cite{FM09}. A.~Ayzenberg \cite{Ayz10} proved
that $s(\Gamma)=m-\lceil\log_2(\gamma(\Gamma)+1)\rceil$ for any graph
$\Gamma$. For the theory of the Buchstaber invariant see
\cite{Ayz10,Ayz11,Er08,Er09,Er11}. In this article we develop the idea that
appears after reading \cite{FM09}: to consider the problem from the viewpoint
of the set of minimal non-simplices of $K$.

I'm grateful to Victor M.~Buchstaber for the discussion of the results of
this paper. During the discussion he suggested to consider the following
modification of his problem.

\vspace{2mm}

\noindent{\bf Problem 2*.} For any $r$ to find a combinatorial criterion for
the simplicial complex $K$ to have $s(K)=r$.
\section{Combinatorial descriptions}
\subsection{Minimal non-simplices}
The set $\omega\subset[m]$ is called a \emph{non-simplex}, if $\omega\notin
K$. Non-simplex $\omega$ is \emph{minimal}, if it's any proper subset belongs
to $K$. Denote by $N(K)$ the set of all minimal non-simplices. We have
$\sigma\in K$ if and only if it does not contain any $\omega\in N(K)$,
therefore $N(K)$ determines $K$ in a unique way.

Minimal non-simplex description of $K$ is convenient for many reasons. For
example, $K$ is a simplex itself if and only if $N(K)=\varnothing$. $K$ is
\emph{flag} if and only if $|\omega|=2$ for any $\omega\in N(K)$. The
\emph{Stanley-Reisner ring} is also defined in these terms:
$$
\mathbb Z[K]=\mathbb Z[\ib{v}_1,\dots,\ib{v}_m]/(v_{i_1}\dots v_{i_k}\colon \{i_1,\dots,i_k\}\in N(K)).
$$
It was proved in \cite{Er09} that if $[m]=\omega_1\cup\dots\cup\omega_l$,
where $\omega_i$ are non-simplices, then
$$
s(K)\geqslant m-\sum\limits_{i=1}^l|\omega_i|+l.
$$
In particular, $s(K)\geqslant l$, if the non-simplices are pairwise disjoint.
\subsection{Buchstaber invariant}
Any subgroup $H\subset T^m$, $H\simeq T^k$, can be described in two dual ways:\\
1) Parametrically:
$$
H=\{(e^{2\pi i
(S_1^1\psi_1+\dots+S_1^k\psi_k)},\dots,e^{2\pi i(S_m^1\psi_1+\dots+S_m^k\psi_k)})\colon
(\psi_1,\dots,\psi_k)\in\mathbb R^k/\mathbb Z^k\},
$$
where $S=\{S_i^j\}\in{\rm Mat}_{m\times k}(\mathbb Z)$, and
$S$ has $k$ units on the diagonal in the canonical form.\\
2) As a kernel of the mapping $T^m\to T^{m-k}$:
$$
(e^{2\pi i \phi_1},\dots,e^{2\pi i\phi_m})\to (e^{2\pi i(\Lambda_1^1 \phi_1+\dots+\Lambda_1^m
\phi_m)},\dots, e^{2\pi i(\Lambda_{m-k}^1 \phi_1+\dots+\Lambda_{m-k}^m \phi_m)}),
$$
where $\Lambda=\{\Lambda_i^j\}\in{\rm Mat}_{(m-k)\times m}(\mathbb Z)$, and
$\Lambda$ has $m-k$ units on the diagonal in the canonical form.

These two descriptions, and two matrices $S$ and $\Lambda$ fit into the exact
sequences:
$$
\begin{CD}
\{1\} @>>>T^k@>>>T^m@>>>T^{m-k}@>>>\{1\}\\
0@>>>\mathbb Z^k@>S>>\mathbb Z^m@>\Lambda>>\mathbb Z^{m-k}@>>>0
\end{CD}
$$
The subgroup $H$ acts freely if and only if $H\cap T^m_{\x}=\{1\}$ for all
$\x\in \Z_K$. It is enough to consider the points $\x$ such that the simplex
$\sigma(\x)$ is maximal. Since $T^m_{\x}=T^{\sigma(\x)}$, we obtain that $H$
acts freely if and only if $H\cap T^{\sigma}=\{1\}$ for all maximal simplices
$\sigma\in K$. This leads to two dual combinatorial descriptions of $s(K)$.

Let us make the following notations: \\
$A^i$ -- the $i$-th row of the matrix $A$;\\
$A_j$ -- the $j$-th column of the matrix $A$;\\
$A^{\omega}$ -- the matrix, consisting of the rows $\{A^i\colon
    i\in\omega\}$;\\
$A_{\omega}$ -- the matrix, consisting of the columns $\{A_j\colon
    j\in\omega\}$;\\
$A^{\widehat\sigma}$ -- the matrix, obtained from $A$ by deletion
    of the rows $\{A^i\colon i\in\omega\}$;\\
$A_{\widehat\sigma}$ -- the matrix, obtained from $A$ by deletion
    of the columns $\{A_j\colon j\in\omega\}$.

\begin{prop}[\cite{BP02,Er09}]\label{spcomb}
\begin{itemize}
We have:
\item[(A)] $s(K)$ is the maximal $k$ that admits a matrix
    $S\in\mat_{m\times k}(\mathbb Z)$ satisfying the condition: for any
    maximal simplex $\sigma\in K$, $|\sigma|=r$, the columns of the
    matrix $S^{\widehat\sigma}$ form part of a basis in $\mathbb Z^{m-r}$
    (equivalently\footnote{This condition was used by HyunWoong Cho and
    JinHong Kim}, the rows $\{S^i\colon i\notin \sigma\}$ span $\mathbb
    Z^k$);
\item[(B)] $s(K)$ is the maximal $k$ that admits a matrix
    $\Lambda\in\mat_{(m-k)\times m}(\mathbb Z)$ satisfying the condition:
    for any maximal simplex $\sigma\in K$, $|\sigma|=r$, the columns
    $\{\Lambda_j\colon j\in\sigma\}$ form part of a basis in $\mathbb
    Z^{m-k}$ (equivalently, the rows of the matrix $\Lambda_{\sigma}$
    span $\mathbb Z^r$);
\end{itemize}
\end{prop}
Similarly in the real case.
\begin{prop}[\cite{BP02,Er09}]\label{sp2comb}
We have
\begin{itemize}
\item[(A2)] $\s(K)$ is the maximal $k$ that admits a matrix
    $S\in\mat_{m\times k}(\z_2)$ satisfying the condition: for any
    maximal simplex $\sigma\in K$ the columns of the matrix
    $S^{\widehat\sigma}$ are linearly independent
    (equivalently\footnote{This condition in a special case was used in
\cite{FM09}}, the rows $\{S^i\colon i\notin \sigma\}$ span $\z_2^k$);
\item[(B2)] $\s(K)$  is the maximal $k$ that admits a matrix
    $\Lambda\in\mat_{(m-k)\times m}(\z_2)$ satisfying the condition: for
    any maximal simplex $\sigma\in K$, $|\sigma|=r$, the columns
    $\{\Lambda_j\colon j\in\sigma\}$ are linearly independent
    (equivalently, the rows of the matrix $\Lambda_{\sigma}$ span
    $\z_2^r$);
\end{itemize}
\end{prop}
We call the matrix $A$ a \emph{$0/1$-matrix}, if it's entries are zeroes and
units.
\begin{lem}[\cite{Er09}]\label{lem:r23}
For a $0/1$-matrix $A$ of sizes $2\times 2$ or $3\times 3$ equality $\det
A=1\mod 2$ implies that $\det A=\pm1$.
\end{lem}
\begin{proof}
For $2\times 2$-matrices one of the rows should contain $0$, therefore we
come to the case $1\times 1$.

For a $3\times 3$-matrix $A$ if one of the rows has two zeroes, we come back
to the case $2\times 2$. If all the rows are different and have one zero and
two units, then their sum is equal to zero modulo two. Hence up to a
transposition of rows and columns we come to the case
$A=\left(\begin{smallmatrix}1&1&0\\1&0&1\\1&1&1\end{smallmatrix}\right),\;
\det A=-1$.
\end{proof}
Let us mention that for $k\times k$-matrices, $k\geqslant 4$, lemma is not
valid. For even $k$ a counterexample is given by the matrix
$$
A_k=\begin{pmatrix}0&1&1&\dots&1\\
1&0&1&\dots&1\\
&\vdots&&\ddots&\\
1&1&1&\dots&0
\end{pmatrix},\quad\det A_k=(-1)^{k-1}(k-1),
$$
and for odd -- by the matrix
$\left(\begin{smallmatrix}1&0\\0&A_{k-1}\end{smallmatrix}\right)$.

\begin{lem}(\cite{Er09})\label{lem:sr=s}
For $r=1,2,3$ we have: $s(K)\geqslant r$ if and only if $\s(K)\geqslant r$.
\end{lem}
\begin{proof}
If the vectors  $\ib{a}_1,\dots,\ib{a}_r\in\z_2^l$ are linearly independent,
then there is an $r\times r$-minor equal to $1$. It follows from lemma
\ref{lem:r23} that it is equal to $\pm 1$ over $\mathbb Z$, therefore these
vectors form part of a basis in $\mathbb Z^l$. Hence the $0/1$-matrix $S$ for
$\s(K)$ satisfies condition (A) for $s(K)$, therefore $\s(K)\geqslant r$
implies $s(K)\geqslant r$.  The opposite implication follows from the fact
that $s(K)\leqslant \s(K)$.
\end{proof}
\subsection{Description of the Buchstaber invariant in terms of minimal non-simplices}
\begin{prop}\label{nonsp} Condition (A) is equivalent to following condition (A*): for any prime
$p$ and any nonzero vector $\ib{a}\in\mathbb \z_p^k$ there exists
$\omega(\ib{a})\in N(K)$ such that $\langle\ib{a},S^i\rangle\ne 0\mod p$ for
all $i\in\omega(\ib{a})$.
\end{prop}
\begin{proof}
Let condition (A) hold but (A*) fail. Then there exists prime $p$ and
$\ib{a}\in\z_p^k\setminus \{0\}$ such that for any $\omega\in N(K)$ there is
$i_{\omega}$ with $\langle \ib{a}, S^{i_{\omega}}\rangle=0\mod p$. Set
$\sigma=[m]\setminus\{i_{\omega}\colon\omega\in N(K)\}$. Then $\sigma\ne
\varnothing$, since $\{S^i,i\in [m]\}$ span $\z^k$. Moreover, $\sigma\in K$.
Otherwise there is $\omega\in N(K)$ such that $\omega\subset\sigma$,
therefore $i_{\omega}\in\sigma$, which is a contradiction. Then all the rows
$\{S^i\colon i\notin\sigma\}$ lie in the proper subgroup $\{\ib{x}\colon
\langle\ib{a},\ib{x}\rangle=0\mod p\}\subset\z^k$. This contradicts to the
fact that they span $\z^k$.

Now let condition (A*) hold but (A) fail. Then for some $\sigma\in K$ the
rows $\{S^i\colon i\notin\sigma\}$ do not span $\z^k$. Therefore the matrix
$S^{\widehat\sigma}$ in the canonical form has nonnegative number $c\ne 1$ on
the diagonal. This means that there exists a primitive vector $\ib{b}$ in
$\z^k$ such that $\langle \ib{b}, S^i\rangle$ is either $0$ (if $c=0$), or is
divided by $c$ (if $c>0$) for all $i\notin\sigma$. Set $p=2$, if $c=0$; or
any prime divisor of $c$, if $c>0$. Set $\ib{a}=\ib{b}\mod p$. Then
$\ib{a}\ne 0$, and $\langle \ib{a},S^i\rangle=0\mod p$ for all $i\notin
\sigma$. On the other hand,
$\omega(\ib{a})\cap([m]\setminus\sigma)\ne\varnothing$, therefore for any
$i\in\omega(\ib{a})\setminus\sigma$ we have $\langle
\ib{a},S^i\rangle\ne0\mod p$, which is a contradiction.
\end{proof}

\begin{prop}\label{nonsp2} Condition (A2) is equivalent to the following condition (A2*): for
any nonzero vector $\ib{a}\in\mathbb \z_2^k$ there exists $\omega(\ib{a})\in
N(K)$ such that $\langle\ib{a},S^i\rangle=1$ in $\z_2$ for all
$i\in\omega(\ib{a})$.
\end{prop}
The proof is similar to the previous one. We give it here for the
completeness.
\begin{proof}
Let condition (A2) hold but (A2*) fail. Then there exists
$\ib{a}\in\z_2^k\setminus \{0\}$ such that for any $\omega\in N(K)$ there is
$i_{\omega}$ with $\langle \ib{a}, S^{i_{\omega}}\rangle=0$. Set
$\sigma=[m]\setminus\{i_{\omega}\colon\omega\in N(K)\}$. Then $\sigma\ne
\varnothing$ since $\{S^i,i\in [m]\}$ span $\z_2^k$. Moreover, $\sigma\in K$.
Otherwise there is $\omega\in N(K)$ such that $\omega\subset\sigma$,
therefore $i_{\omega}\in\sigma$, which is a contradiction. Then all the rows
$\{S^i\colon i\notin\sigma\}$ lie in the hyperplane $
\langle\ib{a},\ib{x}\rangle=0$. This contradicts to the fact that they span
$\z_2^k$.

Now let condition (A2*) hold and (A2) fail. Then for some $\sigma$ the rows
$\{S^i\colon i\notin\sigma\}$ do not span $\z_2^k$. Then they lie in some
hyperplane $\langle\ib{a},\ib{x}\rangle=0$. On the other hand,
$\omega(\ib{a})\cap([m]\setminus\sigma)\ne\varnothing$, therefore for any
$i\in\omega(\ib{a})\setminus\sigma$ we have $\langle \ib{a},S^i\rangle=1$,
which is a contradiction.
\end{proof}
Let us call the linear dependence $\ib{a}_1+\dots+\ib{a}_l=0$ in the vector
space $\z_2^k$ minimal if any proper subset of vectors in
$\{\ib{a}_1,\dots,\ib{a}_l\}$ is linearly independent.
\begin{prop}\label{utv:xi}
We have $\s(K)\geqslant k$ if and only if there exists a mapping
$\xi\colon\z_2^k\setminus\{0\}\to N(K)$ such that
$\xi(\ib{a}_1)\cap\dots\cap\xi(\ib{a}_{2r+1})=\varnothing $ for any minimal
linear dependence $\ib{a}_1+\dots+\ib{a}_{2r+1}=0$.
\end{prop}
\begin{proof}
Let $\s(K)\geqslant k$. Set $\xi(\ib{a})=\omega(\ib{a})$. Let
$i\in\xi(\ib{a}_1)\cap\dots\cap\xi(\ib{a}_{2r+1})$. Then $\langle
\ib{a}_j,S^i\rangle=1$ for all $j=1,\dots,2r+1$. Hence, $\langle
\ib{a}_1+\dots+\ib{a}_{2r+1}, S^i\rangle=1$, therefore
$\ib{a}_1+\dots+\ib{a}_{2r+1}\ne0$.

Now let us prove the "if" part.  For any $i\in [m]$ set
$M_i=\{\ib{a}\colon\xi(\ib{a})\ni i\}$. Consider the system of equations
$\{\langle \ib{a},\ib{x}\rangle=1\colon\ib{a}\in M_i\}$. Let
$\ib{a}_1,\dots,\ib{a}_t$ be a maximal linearly independent subset in $M_i$.
Since any minimal dependence in $M_i$ contains even number of vectors, we
obtain $\ib{a}=\ib{a}_{i_1}+\dots+\ib{a}_{i_{2l+1}}$ for any $\ib{a}\in M_i$.
Therefore all the equations are expressed in terms of basic equations, hence
the system has solutions. Let $S^i$ be some of them. Consider the matrix $S$
consisting of rows $S^i$. From construction we have $\langle
\ib{a},S^i\rangle=1$ for any $i\in\xi(\ib{a})$, therefore $\s(K)\geqslant k$.
\end{proof}

\section{Criteria for $s(K)\geqslant 1,2,3$}
Lemma \ref{lem:sr=s} implies that it is enough to consider $\s(K)$. The
following proposition easily follows from proposition \ref{sp2comb} or
proposition \ref{utv:xi}.
\begin{prop}[Condition (S1)]\label{utv:sp>=1}
We have $\s(K)\geqslant 1$ if and only if $N(K)\ne\varnothing$, i.e.
$K\ne\Delta^n$.
\end{prop}
\begin{prop}[Condition (S2)]\label{utv:sp>=2}
We have $\s(K)\geqslant 2$ if and only if $N(K)$ contains one of the subsets
of the form:
\begin{enumerate}
\item $\{\tau_1,\tau_2,\tau_3\}$:\quad
    $\tau_1\cap\tau_2\cap\tau_3=\varnothing$;
\item $\{\tau_1,\tau_2\}$:\quad $\tau_1\cap\tau_2=\varnothing$.
\end{enumerate}
\end{prop}
\begin{proof}
Proposition \ref{utv:xi} implies that condition $\s(K)\geqslant 2$ is
equivalent to the existence of the mapping $\xi\colon\z_2^2\setminus\{0\}\to
N(K)$ with $\xi(1,0)\cap\xi(0,1)\cap\xi(1,1)=\varnothing$. There are two
possibilities:
\begin{enumerate}
\item The mapping is injective. Set $\tau_1=\xi(1,0)$, $\tau_1=\xi(0,1)$,
    $\tau_3=\xi(1,1)$.
\item Exactly two vectors have the same images. Take them as a basis. Set
    $\tau_1=\xi(1,0)=\xi(0,1)$, and $\tau_2=\xi(1,1)$.
\end{enumerate}
This proves the "only if" part and gives the mappings for the "if" part.
\end{proof}

\begin{prop}[Condition (S3)]\label{utv:sp>=3}
We have $\s(K)\geqslant 3$ if and only if $N(K)$ contains one of the subsets
of the form:
\begin{enumerate}
\item $\{\tau_1,\tau_2,\tau_3,\tau_4,\tau_5,\tau_6,\tau_7\}$:\quad
    $\tau_1\cap\tau_2\cap\tau_4=\varnothing$;\quad$\tau_1\cap\tau_3\cap\tau_5=\varnothing$;\quad
    $\tau_1\cap\tau_6\cap\tau_7=\varnothing$;
$$
\tau_2\cap\tau_3\cap\tau_6=\varnothing;\quad\tau_2\cap\tau_5\cap\tau_7=\varnothing;\quad\tau_3\cap\tau_4\cap\tau_7=\varnothing;\quad\tau_4\cap\tau_5\cap\tau_6=\varnothing;
$$
\item $\{\tau_1,\tau_2,\tau_3,\tau_4,\tau_5,\tau_6\}$:\quad
    $\tau_1\cap\tau_3=\varnothing$;\quad$\tau_1\cap\tau_2\cap\tau_4=\varnothing$;\quad
    $\tau_1\cap\tau_2\cap\tau_5=\varnothing$;
$$
\tau_1\cap\tau_4\cap\tau_6=\varnothing;\quad\tau_1\cap\tau_5\cap\tau_6=\varnothing;\quad\tau_2\cap\tau_3\cap\tau_6=\varnothing;\quad\tau_3\cap\tau_4\cap\tau_5=\varnothing;
$$
\item $\{\tau_1,\tau_2,\tau_3,\tau_4,\tau_5\}$:\quad
    $\tau_1\cap\tau_2=\varnothing$;\quad$\tau_1\cap\tau_5=\varnothing;$\quad
    $\tau_1\cap\tau_3\cap\tau_4=\varnothing$;
$$
\tau_2\cap\tau_3\cap\tau_5=\varnothing;\quad\tau_2\cap\tau_4\cap\tau_5=\varnothing;
$$
\item $\{\tau_1,\tau_2,\tau_3,\tau_4\}$:\quad $\tau_1\cap
    (\tau_2\cup\tau_3\cup\tau_4)=\varnothing;\quad\tau_2\cap\tau_3\cap\tau_4=\varnothing$;
\item $\{\tau_1,\tau_2,\tau_3\}$:\quad
    $\tau_1\cap\tau_2=\tau_1\cap\tau_3=\tau_2\cap\tau_3=\varnothing$.
\end{enumerate}
\end{prop}
\begin{proof}
Proposition \ref{utv:xi} implies that condition $\s(K)\geqslant 3$ is
equivalent to the existence of the mapping $\xi\colon\z_2^3\setminus\{0\}\to
N(K)$ such that $\xi(\ib{a})\cap\xi(\ib{b})\cap\xi(\ib{c})=\varnothing$ for
any triple of pairwise distinct vectors $\ib{a},\ib{b},\ib{c}$ with
$\ib{a}+\ib{b}+\ib{c}=0$. There are exactly $7$ such triples and they
correspond to two-dimensional subspaces:
\begin{gather*}
(1,0,0)+(0,1,0)+(1,1,0)=(1,0,0)+(0,0,1)+(1,0,1)=(1,0,0)+(0,1,1)+(1,1,1)=\\
(0,1,0)+(0,0,1)+(0,1,1)=(0,1,0)+(1,0,1)+(1,1,1)=(0,0,1)+(1,1,0)+(1,1,1)=\\
(1,1,0)+(1,0,1)+(0,1,1)=0.
\end{gather*}
Set $\ib{a}_1=(1,0,0)$, $\ib{a}_2=(0,1,0)$, $\ib{a}_3=(0,0,1)$,
$\ib{a}_4=(1,1,0)$, $\ib{a}_5=(1,0,1)$, $\ib{a}_6=(0,1,1)$,
$\ib{a}_7=(1,1,1)$. Then
\begin{gather*}
\a_1+\a_2+\a_4=0;\quad\a_1+\a_3+\a_5=0;\quad\a_1+\a_6+\a_7=0;\quad\a_2+\a_3+\a_6=0;\\
\a_2+\a_5+\a_7=0;\quad\a_3+\a_4+\a_7=0;\quad\a_4+\a_5+\a_6=0.
\end{gather*}
Now the proof is obtained by enumeration of all the possible cases. In each
case we choose a basis $\e_1,\e_2,\e_3$ in $\z_2^3$ and denote the sets in
the image by $\o_\sigma$, where $\sigma=\{i\colon\xi(\a_i)=\omega_\sigma\}$.
We use the fact that no four vectors can have the same image. Redundant
equalities are enclosed in brackets.

{\bf I.} There are no triples of vectors with the same image. Then there are
at most three pairs of vectors with the same images.
\begin{enumerate}
\item There are no pairs. Then $\Im\xi=\{\o_i\colon i=1,\dots,7\}$, and
\begin{gather*}
\o_1\cap\o_2\cap\o_4=\varnothing;\quad\o_1\cap\o_3\cap\o_5=\varnothing;\quad\o_1\cap\o_6\cap\o_7=\varnothing;\quad\o_2\cap\o_3\cap\o_6=\varnothing;\\
\o_2\cap\o_5\cap\o_7=\varnothing;\quad\o_3\cap\o_4\cap\o_7=\varnothing;\quad\o_4\cap\o_5\cap\o_6=\varnothing.
\end{gather*}
Set $\tau_i=\omega_i$.

\item There is exactly one pair. Choose $\e_1$ and $\e_2$ to be the
    vectors of this pair, and $\e_3\notin\L\{\e_1,\e_2\}$. Then
    $\Im\xi=\{\o_{12},\o_3,\o_4,\o_5,\o_6,\o_7\}$, and
\begin{gather*}
\o_{12}\cap\o_4=\varnothing;\quad\o_{12}\cap\o_3\cap\o_5=\varnothing;\quad\o_{12}\cap\o_6\cap\o_7=\varnothing;\quad\o_{12}\cap\o_3\cap\o_6=\varnothing;\\
\o_{12}\cap\o_5\cap\o_7=\varnothing;\quad\o_3\cap\o_4\cap\o_7=\varnothing;\quad\o_4\cap\o_5\cap\o_6=\varnothing.
\end{gather*}
Set $\tau_1=\omega_{12}$, and $\tau_i=\o_{i+1}$ for $i\geqslant 2$.
\item There are exactly two pairs.
\begin{enumerate}
\item One pair contains the sum of the vectors of the other. Choose
    $\e_1$ and $\e_2$ to be the vectors of the second pair, and
    $\e_3$ to be the vector paired to $\e_1+\e_2$. Then
    $\Im\xi=\{\o_{12},\o_{34},\o_5,\o_6,\o_7\}$, and
\begin{gather*}
\o_{12}\cap\o_{34}=\varnothing;\quad \langle\o_{12}\cap\o_{34}\cap\o_5=\varnothing\rangle;\quad\o_{12}\cap\o_6\cap\o_7=\varnothing;\quad\langle\o_{12}\cap\o_{34}\cap\o_6=\varnothing\rangle;\\
\o_{12}\cap\o_5\cap\o_7=\varnothing;\quad\o_{34}\cap\o_7=\varnothing;\quad\o_{34}\cap\o_5\cap\o_6=\varnothing.
\end{gather*}
Set $\tau_1=\o_{34}$, $\tau_2=\o_5$, $\tau_3=\o_6$, $\tau_4=\o_7$,
    $\tau_5=\o_{12}$.

\item No pair contains the sum of the vectors of the other. Choose
    $\e_1$ and $\e_2$ to be the vectors of the first pair, and $\e_3$
    to be any vector of the second. The only possible case is:
    $\xi(0,0,1)=\xi(1,1,1)$. Then
    $\Im\xi=\{\o_{12},\o_{37},\o_4,\o_5,\o_6\}$, and
\begin{gather*}
\o_{12}\cap\o_4=\varnothing;\quad\o_{12}\cap\o_{37}\cap\o_5=\varnothing;\quad\o_{12}\cap\o_6\cap\o_{37}=\varnothing;\quad\langle\o_{12}\cap\o_{37}\cap\o_6=\varnothing\rangle;\\
\langle\o_{12}\cap\o_5\cap\o_{37}=\varnothing\rangle;\quad\o_{37}\cap\o_4=\varnothing;\quad\o_4\cap\o_5\cap\o_6=\varnothing.
\end{gather*}
Set $\tau_1=\o_4$, $\tau_2=\o_{12}$, $\tau_3=\o_5$, $\tau_4=\o_6$,
$\tau_5=\o_{37}$.
\end{enumerate}
\item There are exactly three pairs.
\begin{enumerate}
\item The seventh vector is not equal to the sum of the vectors of
    any pair. Choose $\e_1$ and $\e_2$ to be the vectors of any pair,
    and $\e_3$ to be the vector paired to $\e_1+\e_2$. We have
    $\xi(0,0,1)=\xi(1,1,0)$. Then the vector
    $(1,1,1)=(0,0,1)+(1,1,0)$ belongs to the third pair. The
    remaining vector of it's pair up to the transposition of $\e_1$
    and $\e_2$ is $(1,0,1)$: $\xi(1,0,1)=\xi(1,1,1)$. Then $\Im
    \xi=\{\o_{12},\o_{34},\o_{57},\o_6\}$, and
\begin{gather*}
\o_{12}\cap\o_{34}=\varnothing;\quad\langle\o_{12}\cap\o_{34}\cap\o_{57}=\varnothing\rangle;\quad\langle\o_{12}\cap\o_6\cap\o_{57}=\varnothing\rangle;\quad\langle\o_{12}\cap\o_{34}\cap\o_6=\varnothing\rangle;\\
\o_{12}\cap\o_{57}=\varnothing;\quad\o_{34}\cap\o_{57}=\varnothing;\quad\langle\o_{34}\cap\o_{57}\cap\o_6=\varnothing\rangle.
\end{gather*}
Set $\tau_1=\o_{12}$, $\tau_2=\o_{34}$, $\tau_3=\o_{57}$.

\item The seventh vector is the sum of two vectors of exactly one
    pair. Choose $\e_1$ and $\e_2$ to be the vectors of this pair,
    and $\ib{e}_3$ to be any of the vectors of the remaining two
    pairs. The vector paired to $\e_3$ can not be $(1,1,1)$,
    therefore up to a transposition of $\e_1$ and $\e_2$ we obtain:
    $\xi(0,0,1)=\xi(1,0,1)$, $\xi(0,1,1)=\xi(1,1,1)$. Then
    $\Im\xi=\{\o_{12},\o_{35},\o_4,\o_{67}\}$, and
\begin{gather*}
\o_{12}\cap\o_4=\varnothing;\quad\o_{12}\cap\o_{35}=\varnothing;\quad\o_{12}\cap\o_{67}=\varnothing;\quad\langle\o_{12}\cap\o_{35}\cap\o_{67}=\varnothing\rangle;\\
\langle\o_{12}\cap\o_{35}\cap\o_{67}=\varnothing\rangle;\quad\o_{35}\cap\o_4\cap\o_{67}=\varnothing;\quad\langle\o_4\cap\o_{35}\cap\o_{67}=\varnothing\rangle.
\end{gather*}
Set $\tau_1=\o_{12}$, $\tau_2=\o_{35}$, $\tau_3=\o_4$,
$\tau_4=\o_{67}$.

\item The seventh vector is the sum of vectors of at least two pairs.
    Choose $\e_1$ and $\e_2$ to be the vectors of the first pair, and
    $\ib{e}_3$ to be any of the vectors of the second. For the second
    pair we obtain: $\xi(0,0,1)=\xi(1,1,1)$, and for the third pair:
    $\xi(1,0,1)=\xi(0,1,1)$. Then
    $\Im\xi=\{\o_{12},\o_{37},\o_4,\o_{56}\}$, and
\begin{gather*}
\o_{12}\cap\o_4=\varnothing;\quad\o_{12}\cap\o_{37}\cap\o_{56}=\varnothing;\quad\langle\o_{12}\cap\o_{56}\cap\o_{37}=\varnothing\rangle;\quad\langle\o_{12}\cap\o_{37}\cap\o_{56}=\varnothing\rangle;\\
\langle\o_{12}\cap\o_{56}\cap\o_{37}=\varnothing\rangle;\quad\o_{37}\cap\o_4=\varnothing;\quad\o_4\cap\o_{56}=\varnothing.
\end{gather*}
Set $\tau_1=\o_4$, $\tau_2=\o_{12}$, $\tau_3=\o_{37}$,
$\tau_4=\o_{56}$.
\end{enumerate}
\end{enumerate}

{\bf II.} There is a triple of vectors with the same image. These vectors are
linearly independent and we can choose them as $\e_1,\e_2$, and $\e_3$.
Consider the rest four vectors.
\begin{enumerate}
\item There are no triples of vectors with the same image.
\begin{enumerate}
\item There are no pairs of vectors with the same image. Then
    $\Im\xi=\{\omega_{123},\omega_4,\o_5,\o_6,\o_7\}$, and
    \begin{gather*}
    \o_{123}\cap\o_4=\varnothing;\quad\o_{123}\cap\o_5=\varnothing;\quad\langle\o_{123}\cap\o_6\cap\o_7=\varnothing\rangle;\quad\o_{123}\cap\o_6=\varnothing;\\
    \langle\o_{123}\cap\o_5\cap\o_7=\varnothing\rangle;\quad\langle\o_{123}\cap\o_4\cap\o_7=\varnothing\rangle;\quad\o_4\cap\o_5\cap\o_6=\varnothing.
    \end{gather*}
    Set $\tau_1=\o_{123}$, $\tau_2=\o_4$, $\tau_3=\o_5$,
    $\tau_4=\o_6$.
\item There is exactly one pair of vectors with the same image.
    \begin{enumerate}
    \item One of the vectors of the pair is $\e_1+\e_2+\e_3$. Up
        to a transposition of $\e_1,\e_2$, and $\e_3$ we obtain:
        $\xi(1,1,0)=\xi(1,1,1)$. Then
        $\Im\xi=\{\o_{123},\o_{47},\o_5,\o_6\}$, and
    \begin{gather*}
    \o_{123}\cap\o_{47}=\varnothing;\quad\o_{123}\cap\o_5=\varnothing;\quad\langle\o_{123}\cap\o_6\cap\o_{47}=\varnothing\rangle;\quad\o_{123}\cap\o_6=\varnothing;\\
    \langle\o_{123}\cap\o_5\cap\o_{47}=\varnothing\rangle;\quad\langle\o_{123}\cap\o_{47}=\varnothing\rangle;\quad\o_{47}\cap\o_5\cap\o_6=\varnothing.
    \end{gather*}
    Set $\tau_1=\o_{123}$, $\tau_2=\o_{47}$, $\tau_3=\o_5$,
    $\tau_4=\o_6$.

    \item Both vectors of the pair are sums of two basis vectors.
        Up to a transposition of $\e_1,\e_2$, and $\e_3$ we
        obtain: $\xi(1,1,0)=\xi(1,0,1)$. Then
        $\Im\xi=\{\o_{123},\o_{45},\o_6,\o_7\}$, and
        \begin{gather*}
        \o_{123}\cap\o_{45}=\varnothing;\quad\langle\o_{123}\cap\o_{45}=\varnothing\rangle;\quad\langle\o_{123}\cap\o_6\cap\o_7=\varnothing\rangle;\quad\o_{123}\cap\o_6=\varnothing;\\
        \langle\o_{123}\cap\o_{45}\cap\o_7=\varnothing\rangle;\quad\langle\o_{123}\cap\o_{45}\cap\o_7=\varnothing\rangle;\quad\o_{45}\cap\o_6=\varnothing.
        \end{gather*}
        Set $\tau_1=\o_{123}$, $\tau_2=\o_{45}$, $\tau_3=\o_6$.
     \end{enumerate}
\item There are exactly two pairs of vectors with the same image. Up
    to a transposition of $\e_1,\e_2$, and $\e_3$ we obtain:
    $\xi(1,1,0)=\xi(1,0,1)$, $\xi(0,1,1)=\xi(1,1,1)$. Then
    $\Im\xi=\{\o_{123},\o_{45},\o_{67}\}$, and
    \begin{gather*}
    \o_{123}\cap\o_{45}=\varnothing;\quad\langle\o_{123}\cap\o_{45}=\varnothing\rangle;\quad\o_{123}\cap\o_{67}=\varnothing;\quad\langle\o_{123}\cap\o_{67}=\varnothing\rangle;\\
    \langle\o_{123}\cap\o_{45}\cap\o_{67}=\varnothing\rangle;\quad\langle\o_{123}\cap\o_{45}\cap\o_{67}=\varnothing\rangle;\quad\o_{45}\cap\o_{67}=\varnothing.
    \end{gather*}
    Set $\tau_1=\o_{123}$, $\tau_2=\o_{45}$, $\tau_3=\o_{67}$.
\end{enumerate}
\item There is a triple of vectors with the same image. Their sum is
    nonzero, therefore up to a transposition of $\e_1,\e_2$, and $\e_3$
    we obtain: $\xi(1,1,0)=\xi(1,0,1)=\xi(1,1,1)$. Then
    $\Im\xi=\{\o_{123},\o_{457},\o_6\}$, and
\begin{gather*}
\o_{123}\cap\o_{457}=\varnothing;\quad\langle\o_{123}\cap\o_{457}=\varnothing\rangle;\quad\langle\o_{123}\cap\o_6\cap\o_{457}=\varnothing\rangle;\quad\o_{123}\cap\o_6=\varnothing;\\
\langle\o_{123}\cap\o_{457}=\varnothing\rangle;\quad\langle\o_{123}\cap\o_{457}=\varnothing\rangle;\quad\o_{457}\cap\o_6=\varnothing.
\end{gather*}
Set $\tau_1=\o_{123}$, $\tau_2=\o_{457}$, $\tau_3=\o_6$.
\end{enumerate}
This enumeration proves the "only if" part. For the "if" part the cases I1,
I2, I3(a), I4(b), and II2 give the mappings for cases 1-5 respectively.
\end{proof}
Now our main result follows from lemma \ref{lem:sr=s} and propositions
\ref{utv:sp>=1},\ref{utv:sp>=2}, and \ref{utv:sp>=3}.
\begin{theo}\label{thmsp}
We have
\begin{enumerate}
\item $s(K)\geqslant 1$ if and only if condition (S1) holds;

\item $s(K)\geqslant 2$ if and only if condition (S2) holds;

\item $s(K)\geqslant 3$ if and only if condition (S3) holds.
\end{enumerate}
\end{theo}
\begin{cor}We have
\begin{enumerate}
\item $s(K)=0$ if and only if $N(K)\ne\varnothing$ (equivalently,
    $K=\Delta^n$);
\item $s(K)=1$ if and only if any two and any three subsets in $N(K)$
    intersect;
\item $s(K)=2$ if and only if some two or three subsets in $N(K)$ do not
    intersect and $N(K)$ does not contain any of $5$ subsets from
    proposition \ref{utv:sp>=3}.
\end{enumerate}
\end{cor}
\section{Problems}
Theorem \ref{thmsp} naturally leads to the following problems.
\begin{problem}
To classify all simplicial complexes $K$ such that $s(K)=2$.
\end{problem}
Minimal non-simplices are closely related to other combinatorial
characteristics of simplicial complexes such as \emph{bigraded Betti numbers}
$$
\beta^{-i,2j}(K)={\rm rank}\,
{\rm Tor}^{-i,2j}_{\mathbb Z[\ib{v}_1,\dots,\ib{v}_m]}(\z[K],\z)={\rm
rank}\,{\rm H}^{-i,2j}[\Lambda[u_1,\dots,u_m]\otimes \z[K],d],
$$
where ${\rm bideg}\, u_i=(-1,2)$, ${\rm bideg}\, v_i=(0,2)$, $du_i=v_i$,
$dv_i=0$. For example, $\sum_j\beta^{-1,2j}=|N(K)|$.
\begin{problem}
To find a criterion for $s(K)=2$ in terms of bigraded Betti numbers.
\end{problem}
Unlike simplicial complexes for a simple $n$-polytope $P$ with $m$ facets
$s(P)=1$ if and only if $P=\Delta^n$ (equivalently, $m-n=1$). The case
$s(P)=2$ is much more complicated. It was shown in \cite{Er09} that
$$
s(C^n(m)^*)=2\quad\text{ for }\quad2\leqslant m-n\leqslant2+\frac{n-13}{48},
$$
where $C^n(m)$ is a cyclic polytope. In particular, for each $k\geqslant 2$
there exists a polytope with $m-n=k$ and $s(P)=2$. Moreover, the estimate
$s(P)\geqslant m-\gamma(P)+s(\Delta^{\gamma-1}_{n-1})$ (see \cite{Er09})
implies that if $s(P)=2$, then one of the following holds:\\
1) $P=I\times I$;\\
2) Any two facets of $P$ intersect, and $m<\frac{7}{4}(n+1)+2$;\\
3) $\gamma(P)=m-1$, and $m<\frac{3}{2}(n+1)+1$.
\begin{problem}
To classify all simple polytopes with $s(P)=2$.
\end{problem}

\end{document}